\documentclass[12pt]{article}
\usepackage{amsmath,amssymb,color,amscd}

\def\seq#1#2#3{#1_{#2},\,\ldots,#1_{#3}}

\def\w{\widetilde}

\def\b{\overline}

\def\tt{{\underline{t}}}
\def\nn{{\underline{n}}}

\def\mm{\underline{m}}

\def\1{\underline{1}}
\def\0{\underline{0}}
\def\R{\mathbb R}
\def\P{\mathbb P}
\def\Q{\mathbb Q}

\def\LLL{\mathbb L}
\def\Z{\mathbb Z}
\def\Q{\mathbb Q}
\def\C{\mathbb C}
\def\S{\mathbb S}
\def\OO{{\mathcal O}}
\def\calE{{\mathcal E}}
\def\XX{{\mathcal X}}
\def\DD{{\mathcal D}}

\def\oD{\stackrel{\circ}{D}}

\def\bD{\stackrel{\bullet}{D}}

\newtheorem{theorem}{Theorem}

\newtheorem{lemma}{Lemma}
\newtheorem{proposition}{Proposition}

\newenvironment{example}
{\smallskip\noindent{\bf Example\/}.}{\smallskip\par}
\newenvironment{examples}
{\smallskip\noindent{\bf Examples\/}.}{\medskip\par}
\newenvironment{remark}
{\smallskip\noindent{\bf Remark\/}.}{\smallskip\par}
\newenvironment{remarks}
{\smallskip\noindent{\bf Remarks\/}.}{\smallskip\par}
\newenvironment{proof}
{\noindent{\bf Proof\/}.}{{ $\Box$}\smallskip\par}

\title{Algebraic links in the Poincar\'e sphere and the Alexander polynomials
\footnote{Math. Subject Class. 14B05, 32S25, 57M25.
Keywords: algebraic links, Poincar\'e sphere, Alexander polynomial,
Poincar\'e series, topological type.
}
}

\author{
A.~Campillo,
\and F.~Delgado,\thanks{The first two authors were supported by the grant
MTM2015-65764-C3-1-P
(with the help of FEDER Program).} \and S.M.~Gusein-Zade
\thanks{
The work of the third author (Sections~\ref{sec:Resolution},
\ref{sec:One_divisorial}, \ref{sec:curve_collection})
was supported by the grant 16-11-10018 of the Russian Science Foundation.
} }

\date{}
\begin{document}
\def\eps{\varepsilon}

\maketitle

\begin{abstract}
The Alexander polynomial in several variables is defined for links in
three-dimensional homology spheres, in particular, in the Poincar\'e sphere: the intersection of the surface
$S=\{(z_1,z_2,z_3)\in\C^3: z_1^5+z_2^3+z_3^2=0\}$ with the 5-dimensional sphere
$\S_{\eps}^5=\{(z_1,z_2,z_3)\in\C^3: \vert z_1\vert^2+\vert z_2\vert^2+\vert z_3\vert^2=\eps^2\}$.
An algebraic link in the Poincar\'e sphere is the intersection of a germ $(C,0)\subset (S,0)$ of a complex
analytic curve in $(S,0)$ with the sphere $\S_{\eps}^3$ of radius $\eps$ small enough.
Here we discuss to which extend the Alexander polynomial in several variables of an
algebraic link in the
Poincar\'e sphere determines the topology of the link.
We show that, if the strict transform of a curve on $(S,0)$ does not
intersect the component of the exceptional divisor
corresponding to the end of the
longest tail in the corresponding $E_8$-diagram, then its Alexander polynomial
determines the combinatorial type of the minimal resolution of the curve and therefore
the topology of the corresponding link.
Alexander polynomial of an algebraic link in the Poincar\'e sphere coincides with the
Poincar\'e series of the filtration defined by the corresponding curve valuations. We
show that, under conditions similar for those for curves, the Poincar\'e series of a
collection of divisorial valuations determines the combinatorial type of the minimal
resolution of the collection.
\end{abstract}

\section{Introduction}\label{sec:Introduction}
The three-dimensional sphere is $\S_{\eps}^3=\{(z_1,z_2)\in\C^2: \vert z_1\vert^2+\vert z_2\vert^2=\eps^2\}$.
An algebraic link in the three-dimensional sphere is the intersection of a germ $(C,0)\subset (\C^2,0)$ of a
complex analytic plane curve with the sphere $\S_{\eps}^3$ with $\eps$ small enough. The number of components
of the link $K=C\cap \S_{\eps}^3$ equals the number of the irreducible components of
the curve $(C,0)$.
A link with $r$ components in the three-sphere has the well-known topological invariant: the Alexander
polynomial in $r$ variables: see, e.~g., \cite{EN}. It is known that the Alexander polynomial in several variables
determines the topological type of an algebraic link (or, equivalently, the (local) topological type of
the triple $(\C^2,C,0)$): \cite{Yamamoto}, see \cite{FAOM} for another proof of this statement.

The Alexander polynomial is defined for links in three-dimensional manifolds which are homology spheres.
The Poincar\'e sphere $\LLL$ is the intersection of the surface $S=\{(z_1,z_2,z_3)\in\C^3: z_1^5+z_2^3+z_3^2=0\}$
with the 5-dimensional sphere
$\S_{\eps}^5=\{(z_1,z_2,z_3)\in\C^3:
\vert z_1\vert^2+\vert z_2\vert^2+\vert z_3\vert^2=\eps^2\}$.
It is a three-dimensional homology sphere.
This definition describes the Poincar\'e sphere $\LLL$ as the link of a rational surface singularity of type
$E_8$. The links of other rational surface singularities are rational homology spheres, but not homology
spheres.

An algebraic link in the Poincar\'e sphere is the intersection of a germ $(C,0)\subset (S,0)$ of a complex
analytic curve in $(S,0)$ with the sphere $\S_{\eps}^5$ of radius $\eps$ small enough. The number of components of
the link $K=C\cap \S_{\eps}^5$ equals the number of the irreducible components of the curve $(C,0)$.
For a link with $r$ components in the Poincar\'e sphere $\LLL=S\cap\S^5_{\eps}$ one has its Alexander polynomial
$\Delta^K(t_1,\ldots,t_r)$ defined in the same way as for a link in the usual three-sphere $\S_{\eps}^3$.

An irreducible curve germ $(C,0)$ in a germ of a complex analytic variety $(V,0)$ defines a valuation $v_C$
on the ring $\OO_{V,0}$ of germ of functions on $(V,0)$ (called a curve valuation).
Let $\varphi:(\C,0)\to(V,0)$ be a parametrization
(an uniformization) of the curve $(C,0)$, that is ${\text{Im\,}}\varphi=(C,0)$ and $\varphi$ is an isomorphism
between punctured neighbourhoods of the origin in $\C$ and in $C$. For a function germ $f\in\OO_{V,0}$,
the value $v_C(f)$ is defined as the degree of the leading term in the Taylor series of the function
$f\circ\varphi:(\C,0)\to\C$:
$$
f\circ\varphi(\tau)=a\tau^{v_C(f)}+{\text{\ terms of higher degree,}}
$$
where $a\ne 0$; if $f\circ\varphi\equiv0$, one defines $v_C(f)$ to be equal to $+\infty$.

A collection $\{(C_i,0)\}$ of irreducible curves in $(V,0)$, $i=1,\ldots,r$, defines the collection $\{v_{C_i}\}$
of valuations. For a collection $\{v_i\}$ of discrete valuations on $\OO_{V,0}$, $i=1,\ldots,r$, there is defined its
Poincar\'e series $P_{\{v_i\}}(t_1,\ldots,t_r)\in\Z[[t_1,\ldots, t_r]]$: \cite{CDK}. In \cite{Duke} it was shown
that, for $(V,0)=(\C^2,0)$, the Poincar\'e series $P_{\{v_{C_i}\}}(t_1,\ldots,t_r)$ of
a collection of (different)
curve valuations coincides with the Alexander polynomial $\Delta^C(t_1,\ldots,t_r)$ in $r$ variables of the algebraic link
defined by the curve $C=\bigcup\limits_{i=1}^rC_i$ for $r>1$. (For $r=1$ one has
$P_{v_C}(t)=\frac{\Delta_C(t)}{1-t}$) In \cite{CMH} it was shown that
the same holds for an algebraic link in the Poincar\'e sphere.

Here we discuss to which extend the Alexander polynomial in several variables of an algebraic link in the
Poincar\'e sphere (that is the Poincar\'e series of the corresponding curve) determines the topology of the link.
Two curves on $(S,0)$ with the same (from the combinatorial point of view) minimal resolutions define topologically
equivalent links in the Poincar\'e sphere.
We show that two curves (even irreducible ones) on $(S,0)$ with
combinatorially different minimal resolutions may have equal Alexander polynomials. The (infinite-dimensional)
space of arcs on $(S,0)$ consists of 8 irreducible components. These components are in one-to-one correspondence
with the components of the exceptional divisor of the minimal (good) resolution of $(S,0)$. A component of the
space of arcs consists of all arcs whose strict transforms intersect the corresponding component of the
exceptional divisor.
We show that, if the strict transform of a (possibly reducible) curve on $(S,0)$ does not
intersect one particular component of the exceptional divisor, namely the one corresponding to the end of the
longest tail in the corresponding $E_8$-diagram, then its Poincar\'e series (that is the Alexander polynomial
of the corresponding link) determines the combinatorial type of the minimal resolution of the curve and therefore
the topology of the corresponding link.

We discuss an analogous question for a collection of divisorial valuations on the $E_8$ surface singularity.
We show that, if no divisor from the collection is born by a sequence of blow-ups starting from a smooth
point of the same component as above, the Poincar\'e series of the collection determines the
combinatorial type of the minimal resolution of the collection of valuations.

The $E_8$ surface singularity is the quotient of the plane $\C^2$ by the binary icosahedral group.
Therefore the results of this paper may have an interpretation it terms of equivariant topology of curves
and/or divisors on the plane with the binary icosahedral group action.

\section{Poincar\'e series of curve and divisorial valuations on the $E_8$-singularity}\label{sec:Resolution}
Let $(S,0)$ be a normal surface singularity of type $E_8$ and let $(C_i,0)$, $i=1,2, \ldots, r$, be
(different) irreducible curves (branches) on $(S,0)$. Let $(C,0)=\bigcup_{i=1}^r(C_i,0)$.
The curves $(C_i,0)$ define curve valuations on the ring $\OO_{S,0}$ of germs of functions on $(S,0)$
in the usual way (see Section~\ref{sec:Introduction}). Let $P_{C}(t_1,\ldots,t_r)$ be
the Poincar\'e series of this set of valuations.

Let $\pi:(\XX,\DD)\to(S,0)$ be an embedded resolution
of the curve $C=\bigcup_{i=1}^r C_i$. This means that:
\begin{enumerate}
 \item[1)] $\XX$ is a smooth surface, $\DD=\pi^{-1}(0)$;
 \item[2)] $\pi$ is a proper complex analytic map;
 \item[3)] the total transform $\pi^{-1}(C)$ of the curve $C$ is a normal crossing divisor on $\XX$ (this
 implies that the exceptional divisor $\DD$ is a normal crossing divisor on $\XX$ as well).
\end{enumerate}
Let $\DD=\bigcup\limits_{\sigma\in\Gamma} D_{\sigma}$ be the decomposition of the exceptional divisor
into irreducible components. All the components $D_{\sigma}$ are isomorphic to the complex projective line
$\C\P^1$. Let $(D_{\sigma}\cdot D_{\delta})$ be the intersection matrix of the components $D_{\sigma}$.
All the self-intersection numbers $D_{\sigma}\cdot D_{\sigma}$ are negative, for $\sigma\ne\delta$
the intersection number
$D_{\sigma}\cdot D_{\delta}$ equals $1$ if the components $D_{\sigma}$ and $D_{\delta}$ intersect each other
and equals $0$ otherwise. The entries of the minus inverse matrix
$(m_{\sigma\delta})=-(D_{\sigma}\cdot D_{\delta})^{-1}$ are positive integers. (Let us recall that we
consider the case of an $E_8$-singularity.) Let $\widetilde{C}_i$ be the strict transform of the branch
$C_i$, i.~e., the closure of the preimage $\pi^{-1}(C_i\setminus\{0\})$, $i=1,2, \ldots, r$,
$\widetilde{C}=\bigcup_{i=1}^r \widetilde{C}_i$. Let $D_{\sigma(i)}$ be the component of the exceptional
divisor $\DD$ intersecting the strict transform $\widetilde{C}_i$ and let
$\mm_{\sigma}:=(m_{\sigma\sigma(1)}, \ldots, m_{\sigma\sigma(r)})\in \Z_{>0}^r$. Let $\oD_{\sigma}$ be the
``smooth part'' of the component $D_{\sigma}$ in the total transform $\pi^{-1}(C)$, i.~e., the component
$D_{\sigma}$ minus the intersection points with all other components of the exceptional divisor and with
the strict transforms $\widetilde{C}_i$. In \cite{CMH} it was shown that
\begin{equation}\label{eq:ACampo}
P_C(\tt)=\prod_{\sigma\in\Gamma}\left(1-\tt^{\mm_{\sigma}}\right)^{-\chi(\oD_{\sigma})},
\end{equation}
where $\tt:=(t_1,\ldots, t_r)$, $\tt^{\mm}:=t_1^{m_1}\cdots t_r^{m_r}$, for $\mm=(m_1,\cdots,m_r)\in\Z^r$,
$\chi(\cdot)$ is the Euler characteristic.

\begin{remarks}
 {\bf 1.} One has an essential difference between Equation~\ref{eq:ACampo} and the corresponding equation for all
 other rational surface singularities from \cite{CMH}. For any other rational surface
 singularity the numbers $m_{\sigma\delta}$ are, generally speaking, not integers and the Poincar\'e series
 is obtained from a certain rational power series  (somewhat similar to
(\ref{eq:ACampo})) in variables $T_{\sigma}$
 corresponding to all the components of the exceptional divisor $\DD$ by eliminating all the monomials with
 non-integer exponents and subsequent substitution of each variable $T_{\sigma}$ by a product of variables
 $t_1$, \dots, $t_r$ corresponding to the branches.

 {\bf 2.} Equation~\ref{eq:ACampo} gives the Poincar\'e series $P_C(\tt)$ in the form
 \begin{equation}\label{eq:Acampo2}
  \prod\limits_{\mm\in (\Z_{\ge0})^r\setminus{\{\0\}}}(1-\tt^{\mm})^{s_{\mm}},
 \end{equation}
 where $s_{\mm}$ are integers. For the $E_8$-singularity this product has finitely many factors. This does
 not hold, in general, for a curve on an arbitrary rational surface singularity. Any power series in the
 variables $t_1$, \dots, $t_r$ with the free term equal to $1$ has a unique representation of the
 form~\ref{eq:Acampo2} (generally speaking, with infinitely many factors).
\end{remarks}

The dual graph of the minimal resolution of the $E_8$-singularity has the standard $E_8$ form (see
Figure~\ref{fig:E_8-res}), all the self-intersection numbers of the components are equal to $-2$.
 An embedded resolution $\pi:(\XX,\DD)\to(S,0)$ of the curve $C$ is obtained from the minimal resolution of the singularity $(S,0)$
by a sequence of blow-ups made (at each step) at intersection points of the strict transform of the
curve $C$ and the exceptional divisor. Some intersection points of the strict transform of the curve $C$
and the exceptional divisor may be at the same time intersection points of components of the exceptional
divisor. Let $\pi':(\XX',\DD')\to(S,0)$ be the resolution of $(S,0)$ obtained only by all the blow-ups at the
points of this sort. The dual graph of the resolution $\pi'$ is of the ``three-tails'' form and is
obtained from the one of the minimal resolution by inserting some (maybe zero) new vertices between the
vertices of the minimal one. The strict transform $\overline{(\pi')^{-1}(C_i\setminus \{0\})}$ of
the branch $C_i$, $i=1, \ldots, r$, intersects the exceptional divisor $\DD'$ at a smooth point of it, i.~e.,
not at an intersection points of its components. Let $\Gamma_0\subset \Gamma$ be the set of indices
$\sigma$ numbering the components of $\DD'$.
\medskip

A divisorial valuation $v$ on $\OO_{S,0}$ is defined by a component of the exceptional divisor $\DD=\pi^{-1}(0)$
of a resolution $\pi:(\XX,\DD)\to(S,0)$ of the singularity. For a function germ $f\in\OO_{S,0}$, the value
$v(f)$ is the multiplicity of the lifting $f\circ\pi$ of the function $f$ to the space $\XX$ of the resolution
along the corresponding component. Let $v_1, \ldots, v_r$ be divisorial
valuations on $\OO_{S,0}$. A resolution of the collection $\{v_i\}$ of divisorial valuations is a
resolution $\pi:(\XX,\DD)\to(S,0)$ whose exceptional divisor contains all the components defining the
valuations. For a resolution $\pi:(\XX,\DD)\to(S,0)$ of the set $\{v_i\}$, let
$\DD=\bigcup\limits_{\sigma\in\Gamma} D_{\sigma}$ be the representation of the exceptional divisor as the
union of its irreducible components, let the integers $m_{\sigma\delta}$ be defined as
above, and let $\bD_{\sigma}$ be the ``smooth part'' of the component $D_{\sigma}$ in the exceptional
divisor $\DD$, i.~e., the component $D_{\sigma}$ minus the intersection points with all other components
of the exceptional divisor. For $i\in\{1, \ldots, r\}$, let $D_{\sigma(i)}$ be the component of the
exceptional divisor $\DD$ defining the divisorial valuation $v_i$. Just as in \cite{CMH} (see also \cite{DG} and
\cite{Invent}), one can show that the Poincar\'e series of the set $\{v_i\}$ of divisorial valuations is given by
\begin{equation}\label{eq:ACampo-divis}
P_{\{v_i\}}(t_1,\ldots, t_r)=\prod_{\sigma\in\Gamma}\left(1-\tt^{\mm_{\sigma}}\right)^{-\chi(\bD_{\sigma})}.
\end{equation}

For a component $D_{\sigma}$ of the exceptional divisor $\DD$ of the resolution $\pi$, let $\ell_{\sigma}$ be
a germ of a smooth curve on $\XX$ transversal to $D_{\sigma}$ at a smooth point of $\DD$ (i.~e., at a point of
$\bD_{\sigma}$), let $(L_{\sigma},0)\subset(S,0)$ be the blow-down $\pi(\ell_{\sigma})$ of the curve $\ell_{\sigma}$
and let $L_{\sigma}$ be given by an equation $h_{\sigma}=0$ with $h_{\sigma}\in\OO_{S,0}$. (Let us recall that an
arbitrary curve germ on the $E_8$ surface singularity is Cartier, i.~e., is defined by an equation.) The curve
germ $L_{\sigma}$ and/or the function germ $h_{\sigma}$ are called {\em a curvette} at the component $D_{\sigma}$.

\section{The Poincar\'e polynomial of an irreducible curve and the topological type}\label{sec:One_branch}
For a plane valuation centred at the origin (say, for a curve or for a divisorial one) the Poincar\'e
series determines the combinatorial type of the minimal resolution: \cite{FAOM}. This does not hold, in general,
for a valuation on a surface singularity. The problem is partially related with the following one.
A resolution of a valuation (a curve or a divisorial one) on a surface singularity $(S,0)$ is at the
same time a resolution of the singularity itself. The minimal resolution of the valuation starts from
a certain point on the exceptional divisor of the minimal resolution of the surface. Therefore a possibility
to determine the combinatorial type of the minimal resolution of a valuation from its Poincar\'e series
assumes that it is possible to determine the component (or the intersection of two components) of the
exceptional divisor of the (minimal) resolution of the surface from which the resolution of the
valuation starts (up to possible symmetries of the dual graph of the minimal resolution of the surface).
However, in general this is not possible.

\begin{example}
 The exceptional divisor of the minimal resolution of the $A_k$ surface singularity consists of $k$
 irreducible components. One can show that a curvette at each of these components is smooth. This follows,
 e.~g., from the results of \cite{L-J_G-S}. Also this can be deduced from the computation of the Poincar\'e series
 of the curvettes using \cite[Theorem~2]{CMH} (which gives $P(t)=\frac{1}{1-t}$).
\end{example}

It appears that the same problem can be met for valuations on the surface singularity of type $E_8$.

\begin{examples}
{\bf 1.} The dual graph of the minimal resolution of the $E_8$-singularity is shown on Figure~\ref{fig:E_8-res}.
\begin{figure}[h]
$$
\unitlength=0.50mm
\begin{picture}(80.00,40.00)(0,10)
\thinlines
\put(-15,30){\line(1,0){90}}
\put(-15,30){\circle*{2}}
\put(0,30){\circle*{2}}
\put(15,30){\circle*{2}}
\put(30,30){\circle*{2}}
\put(45,30){\circle*{2}}
\put(60,30){\circle*{2}}
\put(75,30){\circle*{2}}
\put(15,30){\line(0,-1){15}}
\put(15,15){\circle*{2}}
\put(-16,33){{\scriptsize$1$}}
\put(-1,33){{\scriptsize$2$}}
\put(14,33){{\scriptsize$3$}}
\put(29,33){{\scriptsize$5$}}
\put(44,33){{\scriptsize$6$}}
\put(59,33){{\scriptsize$7$}}
\put(74,33){{\scriptsize$8$}}
\put(17,15){{\scriptsize$4$}}
\end{picture}
$$
\caption{The dual graph of the $E_8$-singularity.}
\label{fig:E_8-res}
\end{figure}
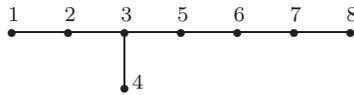
The exceptional divisor consists of 8 irreducible components $D_1$, \dots, $D_8$ numbered as
in the order shown on
Figure~\ref{fig:E_8-res}. Let $C'$ be a curvette at the component $D_6$ and let $C''$ be the blow down of
the plane curve singularity of type $A_4$ (that is with a local equation $u^5+v^2=0$) transversal to the
component $D_8$ at a smooth point of the exceptional divisor (that is, not at the intersection point of
$D_8$ with $D_7$). The minimal resolution of the curve $C'$ coincides with the minimal resolution of the
surface. The dual graph of the minimal resolution of the curve $C''$ is shown on Figure~\ref{fig:L_2-res}.
\begin{figure}
$$
\unitlength=0.50mm
\begin{picture}(120.00,70.00)(0,-35)
\thinlines
\put(-15,30){\line(1,0){90}}
\put(-15,30){\circle*{2}}
\put(0,30){\circle*{2}}
\put(15,30){\circle*{2}}
\put(30,30){\circle*{2}}
\put(45,30){\circle*{2}}
\put(60,30){\circle*{2}}
\put(75,30){\circle*{2}}
\put(15,30){\line(0,-1){15}}
\put(15,15){\circle*{2}}
\put(-16,33){{\scriptsize$1$}}
\put(-1,33){{\scriptsize$2$}}
\put(14,33){{\scriptsize$3$}}
\put(29,33){{\scriptsize$5$}}
\put(44,33){{\scriptsize$6$}}
\put(59,33){{\scriptsize$7$}}
\put(74,33){{\scriptsize$8$}}
\put(17,15){{\scriptsize$4$}}
\put(75,30){\line(0,-1){45}}
\put(75,15){\circle*{2}}
\put(75,0){\circle*{2}}
\put(75,-15){\circle*{2}}
\put(77,15){{\scriptsize$9$}}
\put(77,0){{\scriptsize$10$}}
\put(77,-15){{\scriptsize$12$}}
\put(75,-15){\line(-1,0){15}}
\put(60,-15){\circle*{2}}
\put(57,-13){{\scriptsize$11$}}
\put(75,-15){\vector(1,-1){11}}
\end{picture}
$$
\caption{The dual graph of the minimal resolution of the curve $L_2$.}
\label{fig:L_2-res}
\end{figure}
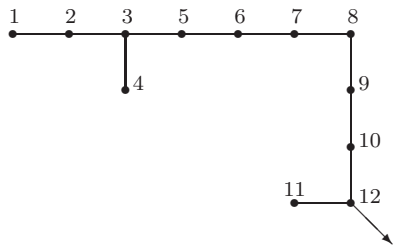
Using (\ref{eq:ACampo}) one can show that
$$
P_{C'}(t)=P_{C''}(t)=\frac{(1-t^{12})(1-t^{18})}{(1-t^{4})(1-t^{6})(1-t^{9})}\,.
$$
Thus the Poincar\'e series of a curve valuation on the $E_8$-singularity does not determine the combinatorial
type of its minimal resolution. Moreover, it does not determine the component of the exceptional divisor
of the minimal resolution of the surface singularity intersecting the strict transform of the curve. (We do not
know whether or not the (algebraic) links in the Poincar\'e sphere corresponding to the curves $L_1$ and $L_2$
are topologically equivalent.)

{\bf 2.} Let $D'$ be the divisor born under the blow-up of the component $D_{12}$ of the resolution
shown on Figure~\ref{fig:L_2-res} (at a smooth point of the exceptional divisor) and let $D''$ be the divisor
born after 7 blow-ups starting at a smooth point of the component $D_6$ and produced at each step at a
smooth point of the previously born divisor. Let $\nu'$ and $\nu''$ be the divisorial valuations defined by
the divisors $D'$ and $D''$ respectively. Using~(\ref{eq:ACampo-divis}) one can show that
$$
P_{\nu'}(t)=P_{\nu''}(t)=\frac{(1-t^{12})(1-t^{18})}{(1-t^{4})(1-t^{6})(1-t^{9})(1-t^{19})}\,.
$$
Thus the Poincar\'e series of a divisorial valuation on the $E_8$-singularity does not determine the
combinatorial type of its minimal resolution. Other examples of this sort can be obtained by applying
the same additional modifications at smooth points of the divisors $D'$ and $D''$.
\end{examples}

The examples show that one cannot restore, in general, the combinatorial type of the (minimal) resolution
of a valuation (say, of an irreducible curve) on the $E_8$ surface singularity from its Poincar\'e series.
However, often this is the case. According to~\cite{MP} the space of arcs on the $E_8$-singularity consists
of 8 irreducible components. Each of them is the closure of
the subspace of arcs whose strict transforms intersect the
exceptional divisor of the minimal resolution of the surface at one of the components $D_1$, \dots, $D_8$.
Let us denote these spaces of arcs by $\calE_1$, \dots, $\calE_8$ respectively. One can show that only
arcs from $\calE_6$ and $\calE_8$ may have the same Poincar\'e series. Moreover, if one restricts the
consideration only to the arcs not intersecting the component $D_8$ at a smooth point of the exceptional
divisor, one can determine the combinatorial type of the (minimal) resolution of the arc from its Poincar\'e
series. The same holds for a reducible curve: if $(C,0)=\bigcup\limits_{i=1}^r (C_i,0)\subset(S,0)$
and the strict transforms of the branches $C_i$ do not intersect the component $D_8$ of the exceptional
divisor at smooth points (that is, they belong to the union $\bigcup\limits_{i=1}^7 \calE_i$),
the Poincar\'e series of the curve $C$ (in $r$ variables) determines the combinatorial type of its minimal
resolution. We shall show that analogues of these statements hold for divisorial
valuations as well.

Let $(C,0)\subset(S,0)$ be an irreducible curve on the $E_8$ surface singularity $(S,0)$
such that its minimal embedded resolution does not start from a smooth point of the
component $D_8$ of the minimal resolution of $(S,0)$.
Let $P_C(t)$ be the Poincar\'e series of the corresponding (curve) valuation on
$\OO_{S,0}$.
Let us recall that $P_C(t)=\frac{\Delta^C(t)}{1-t}$, where $\Delta^C(t)$ is the Alexander polynomial
of the knot $C\cap \LLL$ in the Poincar\'e sphere $\LLL$.

\begin{theorem}\label{theo:One_curve}
 The Poincar\'e series $P_C(t)$ determines the combinatorial type of the minimal embedded resolution
 of the irreducible curve $C\subset S$ $($and therefore the topological type of the knot $(\LLL,C\cap \LLL)$,
 where $\LLL$ is the link of the $E_8$ surface singularity $(S,0)$, i.~e., the Poincar\'e sphere$)$.
\end{theorem}

\begin{proof}
 Let $\pi:(\XX,\DD)\to(S,0)$ be the minimal embedded resolution of the curve $(C,0)$ and let
 $\pi':(\XX',\DD')\to(S,0)$ be the resolution of the surface singularity $(S,0)$ described in
 Section~\ref{sec:Resolution}, $\DD'=\bigcup\limits_{\sigma\in\Gamma_0}D_{\sigma}$. The resolution
 $\pi'$ either is the minimal resolution of the singularity $(S,0)$, or is obtained from the minimal one
 by blow-ups points inbetween two particular components of it. In the latter case
the dual graph of the
 resolution $\pi'$ is obtained from the one of the minimal resolution by inserting several vertices
 inbetween two neighbouring vertices $D_i$ and $D_j$ of the $E_8$-graph. Each component $D_{\sigma_0}$
 ($\sigma_0\in\Gamma_0$) inbetween $D_i$ and $D_j$ is characterized by a pair of coprime positive integers
 $s_1$ and $s_2$ so that one has $m_{k\sigma_0}=s_1m_{ki}+s_2m_{kj}$ for $1\le k\le 8$.

 \begin{remark}
  If one knows the numbers $i$ and $j$ of the components and the ratio $s_1/s_j$ for the component $D_{\sigma_0}$,
  one knows the resolution $\pi'$ itself.
 \end{remark}

 Either the resolution $\pi$ coincides with the resolution $\pi'$ (then $C$ is a curvette at a component
 $D_{\sigma_0}$ of $\DD'$) or it is obtained from $\pi'$ by a sequence of blow-ups starting at a smooth point
 of a component $D_{\sigma_0}$.

 The matrix $(m_{ij})=-(D_i\cdot D_j)^{-1}$ (minus the inverse of the intersection matrix of the components
 of the minimal resolution of the $E_8$ surface singularity) is equal to
 \begin{equation}\label{eq:inverse_matrix}
 \begin{pmatrix}
  4 & 7 & 10 & 5 & 8 & 6 & 4 & 2\\
  7 & 14 & 20 & 10 & 16 & 12 & 8 & 4\\
  10 & 20 & 30 & 15 & 24 & 18 & 12 & 6\\
  5 & 10 & 15 & 8 & 12 & 9 & 6 & 3\\
  8 & 16 & 24 & 12 & 20 & 15 & 10 & 5\\
  6 & 12 & 18 & 9 & 15 & 12 & 8 & 4\\
  4 & 8 & 12 & 6 & 10 & 8 & 6 & 3\\
  2 & 4 & 6 & 3 & 5 & 4 & 3 & 2
 \end{pmatrix}
 \end{equation}
 For $1\le i\le 8$, let $Q_i=(m_{i8}, m_{i1}, m_{i4})\in \R^3$. For $i$, $j$ such that the components
 $D_i$ and $D_j$ intersect, let $I_{ij}$ be the segment between the points $Q_i$ and $Q_j$ (that is, the set
 of points of the form $\lambda Q_i+(1-\lambda)Q_j$ with $0\le\lambda\le 1$). Let us consider the
 one-dimensional simplicial complex $G$ in $\R^3$ with the vertices $Q_i$ and the edges $I_{ij}$.
 (As an abstract graph $G$ is isomorphic to the $E_8$-graph.)

 \begin{lemma}\label{lemma:planar_graph}
  The image of the graph $G$ in $\R\P^2$ (under the natural quotient map $\R^3\setminus\{0\}\to\R\P^2$)
  is a planar graph, i.~e. images in $\R\P^2$ of different points of $G$ are different.
 \end{lemma}

A proof is obtained by drawing the image of $G$ in an affine chart of $\R\P^2$:
Figure~\ref{fig:ratios}.
\begin{figure}
$$
\unitlength=0.75mm
\begin{picture}(140.00,100.00)(0,0)
\thinlines
\put(10,10){\vector(1,0){120}}
\put(10,10){\vector(0,1){80}}
\put(70,60){\line(1,0){40}}
\put(70,60){\line(0,1){10}}
\put(60,35){\line(2,5){10}}
\put(133,10){{\scriptsize$m_{k8}/m_{k1}$}}
\put(13,90){{\scriptsize$m_{k4}/m_{k1}$}}
\put(10,10){\line(0,-1){3}}
\put(7,2){{\scriptsize$0.0$}}
\put(60,10){\line(0,-1){3}}
\put(57,2){{\scriptsize$0.5$}}
\put(110,10){\line(0,-1){3}}
\put(107,2){{\scriptsize$1.0$}}
\put(10,10){\line(-1,0){3}}
\put(-2,8){{\scriptsize$1.0$}}
\put(10,60){\line(-1,0){3}}
\put(-2,58){{\scriptsize$1.5$}}
\put(10,80){\line(-1,0){3}}
\put(-2,78){{\scriptsize$1.7$}}
\put(60,35){\circle*{1.5}}
\put(55,33){{\scriptsize$1$}}
\put(67.15,52.85){\circle*{1.5}}
\put(62.15,50.85){{\scriptsize$2$}}
\put(70,60){\circle*{1.5}}
\put(65,58){{\scriptsize$3$}}
\put(70,70){\circle*{1.5}}
\put(65,68){{\scriptsize$4$}}
\put(72.5,60){\circle*{1.5}}
\put(70.5,62){{\scriptsize$5$}}
\put(76.67,60){\circle*{1.5}}
\put(74.65,62){{\scriptsize$6$}}
\put(85,60){\circle*{1.5}}
\put(83,62){{\scriptsize$7$}}
\put(110,60){\circle*{1.5}}
\put(108,62){{\scriptsize$8$}}
\end{picture}
$$
\caption{The graph $G$.}
\label{fig:ratios}
\end{figure}
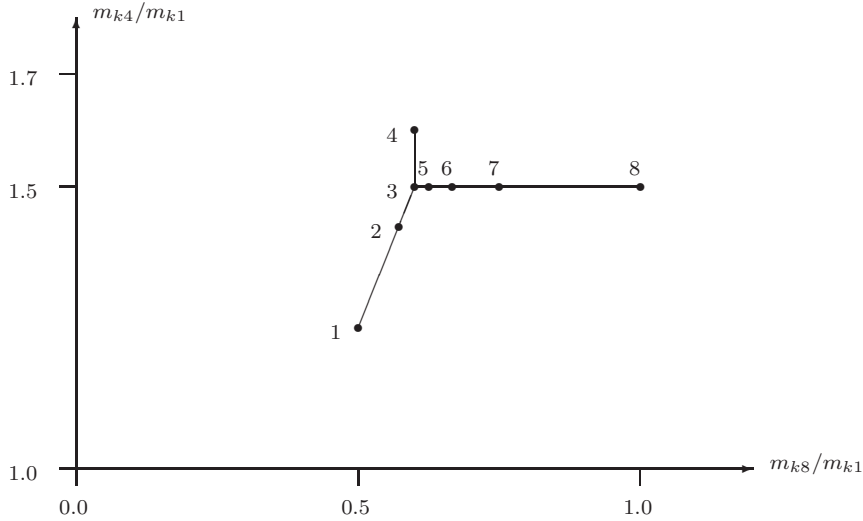

\begin{remark}
 One can see that the graph embedded into the projective plane consists of straight lines in between the rupture
 points and the deadends. This is a general property for the image in the projective
space of the dual resolution graph of
 a surface singularity under the map which sends a vertex $\sigma$ to the ratio of the ``multiplicities''
 $m_{\sigma_i\sigma}$ (that is of the elements of the minus inverse of the intersection matrix) for deadends
 $\sigma_i$ of the graph.
\end{remark}

 Lemma~\ref{lemma:planar_graph} says that the ratios of the three coordinates of different points
 of $G$ never coincide.

 \begin{lemma}\label{lemma:curveD'}
  The Poincar\'e series $P_C(t)$ of the curve $C$ determines the resolution $\pi'$ and the component
  $D_{\sigma_0}$ in $\DD'$ intersecting the strict transform of the curve $C$.
 \end{lemma}

 \begin{proof}
 Let us write the Poincar\'e series $P_C(t)$ in the form
 \begin{equation}\label{eq:curve_binomial_product}
  \prod_{i=1}^q(1-t^{m_i})^{-1}\prod_{m> 0}(1-t^{m})^{s_m},
 \end{equation}
 where $m_1\le m_2\le \ldots\le m_q$ (thus the first product may have repeated factors) and in the second product
 the (integer) exponents $s_m$ are non-negative and are equal to zero for $m=m_i$, $i=1,\ldots,q$. Let us recall
 that the representation of the Poincar\'e series in this form is unique. This follows from the fact that any
 series from $1+t\Z[[t]]$ has a unique representation in the form of a (finite or infinite) product
 $\prod\limits_{m=1}^{\infty}(1-t^m)^{s_m}$ with integer exponents $s_m$.

 The representation~(\ref{eq:curve_binomial_product}) may not have less than two binomial factors with the
 exponent $(-1)$. As a rational function the Poincar\'e series $P_C(t)$ has the form of a polynomial (in fact
 the Alexander polynomial of the corresponding algebraic link) divided by $(1-t)$. One cannot have a degree of
 the binomial $(1-t)$ in~(\ref{eq:curve_binomial_product}) since all the entries of the
 matrix~(\ref{eq:inverse_matrix}) are greater than 1. If there is one factor $(1-t^{m_1})^{-1}$ with $m_1>1$,
 its poles at the degree $m_1$ roots of $1$ different from 1 have to be zeroes of the second product
 $\prod\limits_{m}(1-t^{m})^{s_m}$. This implies that this product contains a binomial $(1-t^{km_1})$ with
 a non-zero exponent $s_{km_1}$ and therefore the series itself is a polynomial.

 If the strict transform $\widetilde{C}$ (in the space $\XX'$ of the resolution $\pi'$) intersects the component
 $D_1$, then the ratio $m_2/m_1$ is greater than $2$. This follows from the fact that
the exponent $m_1$ is equal to
 $\ell m_{18}=2\ell$, where $\ell$ is the intersection number $\widetilde{C}\cdot D_1$, and the exponent $m_2$
 is either equal to $\ell m_{14}=5\ell$ or corresponds to a divisor born from $D_1$ under some blow-ups. In the
 last case it is greater that $\ell m_{11}=4\ell$. The ratio $m_2/m_1>2$ cannot be met in other cases: see below.

 If the strict transform $\widetilde{C}$ does not intersect the component $D_1$ (and intersects a component
 $D_{\sigma_0}$ of the exceptional divisor $\DD$), one has $m_1=\ell m_{\sigma_0 8}$, $m_2=\ell m_{\sigma_0 1}$,
 where $\ell=\widetilde{C}\cdot D_{\sigma_0}$. (This follows from the fact that in the
 matrix~(\ref{eq:inverse_matrix}) one has $m_{i8}\le m_{i1}\le m_{ik}$ for $1<k<8$
and all $i$.) From Figure~\ref{fig:ratios}
 one can see that $1<m_2/m_1<2$. Moreover, if $m_2/m_1\ne 5/3$, the ratio $m_2/m_1$ determines the component
 $D_{\sigma_0}$ (that is the components $D_i$ and $D_j$ from the minimal resolution graph of the $E_8$ surface
 singularity and the corresponding ratio $s_1/s_2$). The ratio $m_2/m_1$ is equal to $5/3$ if and only if the
 strict transform $\widetilde{C}$ intersects either $D_3$, or $D_4$, or a component inbetween $D_3$ and $D_4$
 in $\DD'$. In this case if the strict transform $\widetilde{C}$ does not intersect $D_4$, there are at least
 three binomial factors with the exponent $-1$ and one has $m_1=\ell m_{\sigma_0 8}$, $m_2=\ell m_{\sigma_0 1}$,
 $m_3=\ell m_{\sigma_0 4}$. Lemma~\ref{lemma:planar_graph} implies that the ratio $m_3:m_2:m_1$ determines the
 component $D_{\sigma_0}$. If there are less than three binomial factors with the exponent $-1$ or the ratio
 $m_3:m_2:m_1$ is different from $8:5:3$, the strict transform $\widetilde{C}$ intersects the component $D_4$.
 \end{proof}

 \begin{remark}
 One can avoid arguments from the first paragraph of the proof formulating the criterium for
 the strict transform $\widetilde{C}$ not to intersect the component $D_1$ in the following form:
 this holds if and only if the Poincar\'e series $P_C(t)$ contains at least two binomial factors with
 the exponent $(-1)$ and $m_2/m_1\le 2$. The formal negation says that the strict transform $\widetilde{C}$
 intersects the component $D_1$ if and only if either the Poincar\'e series $P_C(t)$ contains less than two
 binomial factors with the exponent $(-1)$ or $m_2/m_1> 2$. The first option does not take place, but
 this does not contradict the statement. This sort of formulation can be useful for the proof of a version
 of this Lemma for a divisorial valuation in Section~\ref{sec:One_divisorial}.
\end{remark}

Lemma~\ref{lemma:curveD'} says that the Poincar\'e series $P_C(t)$ determines the (minimal) modification
$\pi':(\XX',\DD')\to(S,0)$ and the component $D_{\sigma_0}$ containing the point $P=\widetilde{C}\cap\DD'$
(a smooth point of $\DD'$). In particular, one knows the multiplicity
$m_{\sigma_0\sigma_0}$.
In terms of the decomposition~(\ref{eq:curve_binomial_product}) one has $m_1=\ell m_{8\sigma_0}$ and therefore
the intersection number $\ell=\widetilde{C}\cdot D_{\sigma_0}$ is determined by the Poincar\'e series $P_C(t)$.
Let
$$
D(t):=(1-t^{\ell m_{1\sigma_0}})^{-1}(1-t^{\ell m_{4\sigma_0}})^{-1}(1-t^{\ell m_{8\sigma_0}})^{-1}
(1-t^{\ell m_{3\sigma_0}})\,.
$$

Let $\pi'':(\XX,\DD'')\to(\XX',P)$ be the minimal embedded resolution of the germ
$(\widetilde{C}\cup D_{\sigma_0}, P)\subset (\XX',\DD')$,
$\DD''=\bigcup_{\delta\in\Gamma_1}F_{\delta}$.
The dual graph $\Gamma_1$ (with an arrow representing the strict transform of the curve $\w C$ by $\pi''$)
is of the form shown on Figure~\ref{fig:Felix1}.
\begin{figure}[h]
$$
\unitlength=1.00mm
\begin{picture}(80.00,20.00)(-10,13)
\thinlines
\put(-5,30){\line(1,0){41}}
\put(44,30){\line(1,0){16}} \put(38,30){\circle*{0.5}}
\put(40,30){\circle*{0.5}} \put(42,30){\circle*{0.5}}
\put(30,10){\line(0,1){20}} \put(50,20){\line(0,1){10}}
\put(60,10){\line(0,1){20}} \put(10,15){\line(0,1){15}}
\put(60,30){\vector(1,1){5}}
\put(20,30){\circle*{1}} \put(30,30){\circle*{1}}
\put(50,30){\circle*{1}} \put(60,30){\circle*{1}}

\put(30,20){\circle*{1}}
\put(60,25){\circle*{1}}
\put(60,20){\circle*{1}}
\put(60,15){\circle*{1}}
\put(10,30){\circle*{1}}
\put(30,10){\circle*{1}} \put(50,20){\circle*{1}}
\put(60,10){\circle*{1}}
\put(-5,30){\circle*{1}}
\put(0,30){\circle*{1}} \put(5,30){\circle*{1}}
\put(15,30){\circle*{1}} \put(25,30){\circle*{1}}
\put(35,30){\circle*{1}} \put(45,30){\circle*{1}}
\put(55,30){\circle*{1}} \put(10,25){\circle*{1}}
\put(10,20){\circle*{1}} \put(10,15){\circle*{1}}
\put(30,25){\circle*{1}} \put(30,15){\circle*{1}}
\put(35,30){\circle*{1}} \put(-9,25){$\tau_0$}
\put(11.5,14){$\tau_1$}
\put(31.5,9){$\tau_2$}
\put(61.5,9){$\tau_g$}
\put(9,32){$\delta_1$} \put(29,32){$\delta_2$}
\put(57.5,33){$\delta_g$}
\end{picture}
$$
\caption{The graph $\Gamma_1$.}
\label{fig:Felix1}
\end{figure}
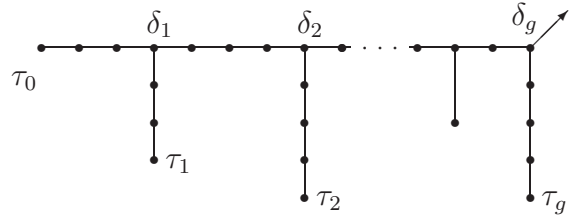
Here $\tau_0$ marks the divisor corresponding to the first blow-up at the point $P$,
$g$ is the number of Puiseux pairs of the curve $\widetilde{C}$ and $F_{\delta_g}$ is the last
component of $\DD''$ (i.~e., the component with self-intersection $-1$): the strict transform of $\w C$
intersects $F_{\delta_g}$; $\tau_i$, $i=0,1,\ldots, g$, are the deadends of the graph
$\Gamma_1$.
Let $(\ell_{\tau_i},P)\subset(\XX',P)$ be a curvette at the component
$D_{\tau_i}$ and let $\b \beta_{i}$ be the intersection number $\w C\cdot\ell_{\tau_i}$.
One has $\b{\beta}_0 <\b{\beta}_1 <\cdots < \b{\beta}_g$ and
$\{\b{\beta}_i\vert i=0,1,\ldots, g\}$ is the minimal set of generators of the semigroup of values of
the curve germ $(\w C,P)$ (in particular $\b{\beta}_0$ is the multuiplicity of $\w C$). Moreover,
the sequence $\seq{\b{\beta}}0g$ determines the graph $\Gamma_1$.

Let
$$
Q(t) := \prod_{j=0}^s (1-t^{m_{\delta_g \tau_j}})^{-1} \ \cdot\
\prod_{j=1}^s (1-t^{m_{\delta_g \delta_j}}) \; .
$$
If $\w C$ is a curvette at the component $D_{\sigma_0}$ one defines $Q(t)$ to be equal to $1$ and
$\Gamma_1$ is only an arrow without any vertex. If $\w C$ is smooth but
$\w C\cdot D_{\sigma_0} = \ell > 1$ (that is $\w C$ is tangent to $D_{\sigma_0}$), we put
$Q(t) = (1-t^{m_{\delta_g\tau_0}})^{-1}$ (the dual graph is shown in Figure~\ref{fig:Case_2}).
\begin{figure}[h]
$$
\unitlength=1.00mm
\begin{picture}(60.00,25.00)(0,0)
\thinlines
\put(5,10){\circle*{1}}
\put(2,8){$1$}
\put(5,10){\line(1,0){45}}
\put(50,10){\circle*{1}}
\put(51,8){$8$}

\put(15,10){\line(0,-1){10}}
\put(15,0){\circle*{1}}
\put(12,0){$4$}

\put(30,10){\circle*{1}}
\put(27,7){$\sigma_0$}
\put(30,10){\line(0,1){10}}
\put(30,20){\circle*{1}}

\put(30,20){\vector(1,1){5}}
\put(27,21){$\delta$}
\put(30,20){\line(-1,0){20}}
\put(10,20){\circle*{1}}
\put(9,15){$\tau_0$}
\put(15,20){\circle*{1}}
\end{picture}
$$
\caption{Case 2}
\label{fig:Case_2}
\end{figure}
Pay attention that
\begin{gather*}
m_{\delta_g\tau_0} < m_{\delta_g\tau_1} < \ldots < m_{\delta_g\tau_g}\,,
\\
m_{\delta_g\tau_j} < m_{\delta_g\delta_j} < m_{\delta_{g} \tau_{j+1}}\,.
\end{gather*}

The modification $\pi = \pi'\circ \pi'' : (\XX',\DD)\to (S,0)$ is the minimal resolution of the
curve $C$. The dual graph of this resolution (i.~e., the one of the modification
$\pi$ with an arrow corresponding to $C$ added) is obtained by joining the graphs $\Gamma_0$ and $\Gamma_1$
by an edge between $\sigma_0$ and a vertex $\delta\in \Gamma_1$.
If $\w C$ is a curvette at the component $D_{\sigma_0}$, the graph $\Gamma$ is obtained from the graph
$\Gamma_0$ by attaching an arrow to the vertex $\sigma_0\in \Gamma_0$. This case (we refer to it as {\bf Case 1}
in the sequel) is characterized by the condition $\ell = 1$. One has
$$
P_C (t) = D(t) (1-t^{m_{\sigma_0 \sigma_0}}) \; .
$$

We will consider several cases corresponding to essentially different possibilities for the position of
the vertex $\delta$ in $\Gamma_1$.
Taking into account the fact that $m_{\delta_g\sigma} = \ell m_{\sigma\sigma}$ for all $\sigma\in \Gamma_0$,
one can show that in all these cases one has
$$
P_C (t) = D(t) Q(t) (1-t^{m_{\delta_g \delta}})(1-t^{m_{\delta_g\sigma_0}})
$$
(see the discussion of the cases below).
The exponent $m_{\delta_g \sigma_0}$ is equal to $\ell m_{\sigma_0 \sigma_0}$ and
the series $D(t)$ is known. Therefore the remaining part of the proof consists in the computation of $\delta$,
$m_{\delta_g\delta}$ and $\seq{\b{\beta}}0g$ from the (known) series
$B(t) = Q(t) (1-t^{m_{\delta_g\delta}})$.

Let us write the series $B(t)$ in the form
$$
B(t) = \prod_{k=1}^{r} (1-t^{m_k})^{-1} \ \cdot \ \prod_{k=1}^{r} (1-t^{n_k})
$$
with $m_1<\cdots < m_r$ and $n_1<\cdots < n_r$. In the following $\mu$ will denote
the integer $\mu = m_1-m_{\delta_g\sigma_0}$.

{\bf Case 2.}
The curve $\w C$ is smooth but $\ell >1$. In this case one has
$\delta = \delta_g$, $m_{\delta_g\tau_0} = \ell m_{\sigma_0\sigma_0} +1$ and
$m_{\delta_g\delta} = \ell (m_{\delta_g\sigma_0} +1 ) > m_{\delta_g\tau_0}$.
The graph $\Gamma_1$ is shown in Figure~\ref{fig:Case_2} and the series $B(t)$ is equal to
$$
B(t) = (1-t^{m_{\delta_g\tau_0}})^{-1} \cdot
(1-t^{\ell m_{\delta_g\tau_0}})\,.
$$
Note that in this case $\mu=1$.

{\bf Case 3.} The curve $\w C$ is non smooth but it is transversal to the component $D_{\sigma_0}$.
In this case $\delta = \tau_0$ and $m_{\delta_g\delta}= m_{\delta_g\tau_0}$.
The binomial factor $(1-t^{m_{\delta_g\delta}}) = (1-t^{m_{\delta_g\tau_0}})$
does not participate in the decomposition of
$B(t)$. The integer $\ell=\b{\beta}_0$ coincides with the multiplicity of $\w C$ at $P$ and
one has
$$
m_1 = m_{\delta_g\tau_1} = \ell m_{\sigma_0\sigma_0} + \b{\beta}_1 =
m_{\delta_g\sigma_0}+ \b{\beta}_1 \; .
$$
For $\mu = m_1-m_{\delta_g,\sigma_0}$ one has $\mu > \ell > 1$.
Comparing with all the other cases below, one can see that these conditions characterize the case 3.
One has $\b{\beta}_0 = \ell$, $\b{\beta}_1 = \mu$ and the graph $\Gamma$ is shown on Figure~\ref{fig:Case_3}.
\begin{figure}[h]
$$
\unitlength=1.00mm
\begin{picture}(80.00,35.00)(0,0)
\thinlines
\put(5,10){\circle*{1}}
\put(1,8){$1$}
\put(5,10){\line(1,0){40}}
\put(45,10){\circle*{1}}
\put(47,8){$8$}

\put(15,10){\line(0,-1){10}}
\put(15,0){\circle*{1}}
\put(11,0){$4$}

\put(30,10){\circle*{1}}
\put(27,7){$\sigma_0$}
\put(30,10){\line(0,1){15}}
\put(30,25){\circle*{1}}


\put(17,23){$\delta=\tau_0$}
\put(30,25){\line(1,0){20}}
\put(40,25){\circle*{1}}
\put(38,27){$\delta_1$}
\put(40,25){\line(0,-1){10}}
\put(42,13){$\tau_1$}

\put(40,15){\circle*{1}}
\put(52,25){\circle*{0.5}}
\put(54,25){\circle*{0.5}}
\put(56,25){\circle*{0.5}}
\put(60,25){\line(1,0){15}}
\put(65,25){\circle*{1}}
\put(65,25){\line(0,-1){10}}
\put(65,15){\circle*{1}}
\put(75,25){\circle*{1}}
\put(75,25){\vector(1,1){5}}
\put(71,27){$\delta_g$}
\put(75,25){\line(0,-1){15}}
\put(75,10){\circle*{1}}
\put(77,8){$\tau_g$}
\end{picture}
$$
\caption{Case 3}
\label{fig:Case_3}
\end{figure}
Simple computations using the well-known description of the resolution process of a plane curve
singularity give us
\begin{align*}
m_{\delta_g\tau_j} &=
N_{j-1}\cdots N_1 m_{\delta_g\sigma_0}+ \b{\beta}_j\,, \\
m_{\delta_g\delta_j} &= N_{j}\cdots N_1  m_{\delta_g\sigma_0}+ N_j \b{\beta}_j =
N_{j} m_{\delta_g\tau_j}
\end{align*}
for $j=1,\ldots, g$. Here $N_k = e_{k-1}/e_k$, where $e_k =\gcd (\seq{\b{\beta}}0k)$ for $k=0,1,\ldots, g$.
Note that $N_k$ depends only on the numbers $\seq{\b{\beta}}0k$.
The equations above permit to compute the sequence $\seq{\b{\beta}}0g$ starting from the (known) sequence
$\b{\beta}_0 < m_{\delta_g\tau_1} < \cdots < m_{\delta_g\tau_g}$.

The expression for $B(t)$ implies that $g=r$, $m_j = m_{\delta_g\tau_j}$ for $j\ge 1$.

{\bf Case 4.} The component $D_{\sigma_0}$ has the maximal contact with $\w C$, i.~e.,
$\ell = \w C\cdot D_{\sigma_0} = \b{\beta}_1$. In this case $\delta = \tau_1$ and
$m_{\delta_g\delta}= m_{\delta_g\tau_1}$. As a consequence the binomial factor
$(1-t^{m_{\delta_g\delta}}) = (1-t^{m_{\delta_g\tau_1}})$ does not participate in
the decomposition of the series $B(t)$.
In this case one has
$$
m_1 = m_{\delta_g\tau_0} = \ell m_{\sigma_0\sigma_0} + \b{\beta}_0 = m_{\delta_g\sigma_0}+ \b{\beta}_0 \; .
$$
This case is characterized by the conditions $\mu = m_1-m_{\delta_g\sigma_0} < \ell $ and $\mu$ does not
divide $\ell$. The graph $\Gamma$ is shown in Figure~\ref{fig:Case_4}.
\begin{figure}[h]
$$
\unitlength=1.00mm
\begin{picture}(80.00,35.00)(0,-5)
\thinlines
\put(5,5){\circle*{1}}
\put(1,3){$1$}
\put(5,5){\line(1,0){40}}
\put(45,5){\circle*{1}}
\put(47,3){$8$}

\put(15,5){\line(0,-1){10}}
\put(15,-5){\circle*{1}}
\put(11,-5){$4$}

\put(30,5){\circle*{1}}
\put(27,2){$\sigma_0$}
\put(30,5){\line(0,1){20}}
\put(30,25){\circle*{1}}


\put(30,25){\line(-1,0){15}}
\put(15,25){\circle*{1}}
\put(10,24){$\tau_0$}
\put(30,13){\circle*{1}}
\put(17,12){$\delta=\tau_1$}
\put(28,27){$\delta_1$}

\put(30,25){\line(1,0){20}}
\put(45,25){\circle*{1}}
\put(43,27){$\delta_2$}
\put(45,25){\line(0,-1){10}}
\put(47,13){$\tau_2$}
\put(45,15){\circle*{1}}

\put(52,25){\circle*{0.5}}
\put(54,25){\circle*{0.5}}
\put(56,25){\circle*{0.5}}
\put(60,25){\line(1,0){15}}
\put(65,25){\circle*{1}}
\put(65,25){\line(0,-1){10}}
\put(65,15){\circle*{1}}
\put(75,25){\circle*{1}}
\put(75,25){\vector(1,1){5}}
\put(71,27){$\delta_g$}
\put(75,25){\line(0,-1){15}}
\put(75,10){\circle*{1}}
\put(77,8){$\tau_g$}
\end{picture}
$$
\caption{Case 4}
\label{fig:Case_4}
\end{figure}
As in the previous case one has
\begin{align*}
m_{\delta_g\tau_j} &=
N_{j-1}\cdots N_2 \frac{\b{\beta}_1}{e_1}  m_{\delta_g\sigma_0}+ \b{\beta}_j\,, \\
m_{\delta_g\delta_j} &= N_{j}\cdots N_2 \frac{\b{\beta}_1}{e_1}
m_{\delta_g\sigma_0}+ N_j \b{\beta}_j
= N_{j} m_{\delta_g\tau_j}
\end{align*}
for $j=1,\ldots, g$. These equations permit to compute the sequence $\seq{\b{\beta}}0g$ starting from the sequence
$m_{\delta_g\tau_2} < \cdots < m_{\delta_g\tau_g}$ and the (known) integers $\mu$ and $\ell$.

The expression for $B(t)$ implies that $g=r$, $m_j = m_{\delta_g\tau_j}$ for $j\ge 2$,
$n_j = m_{\delta_g\delta_j}$ for $j\ge 1$. Thus the integers $\seq{m}1r$ permit to compute
the integers $\seq{\b{\beta}}0g$.

{\bf Case 5.}
The component $D_{\sigma_0}$ is tangent to $\w C$ but does not have the maximal contact with it.
This case is equivalent to the condition $\ell = k\cdot \b{\beta}_0$ for some
integer $k$ with
$1 < k <  \b{\beta}_1 / \b{\beta}_0$. The vertex $\delta$ is the $k$-th vertex of
the geodesic in $\Gamma_1$ from $\tau_0$ to $\delta_1$. In this case one has
$$
m_1 = m_{\delta_g\tau_0} = \ell m_{\sigma_0\sigma_0} + \b{\beta}_0 =
m_{\delta_g\sigma_0}+ \b{\beta}_0 \; .
$$
The case is characterized by the conditions
$\mu = m_1-m_{\delta_g\sigma_0} < \ell $ and $\mu (= \b{\beta}_0)$ divides
$\ell (= k\b{\beta}_0)$.
The graph $\Gamma$ is shown in Figure~\ref{fig:Case_5}.
\begin{figure}[h]
$$
\unitlength=1.00mm
\begin{picture}(80.00,35.00)(0,0)
\thinlines
\put(5,10){\circle*{1}}
\put(1,8){$1$}
\put(5,10){\line(1,0){40}}
\put(45,10){\circle*{1}}
\put(47,8){$8$}

\put(15,10){\line(0,-1){10}}
\put(15,0){\circle*{1}}
\put(11,0){$4$}

\put(30,10){\circle*{1}}
\put(27,7){$\sigma_0$}
\put(30,10){\line(0,1){15}}
\put(30,25){\circle*{1}}


\put(30,25){\line(-1,0){10}}
\put(20,25){\circle*{1}}
\put(15,24){$\tau_0$}
\put(28,27){$\delta$}

\put(30,25){\line(1,0){20}}
\put(40,25){\circle*{1}}
\put(38,27){$\delta_1$}
\put(40,25){\line(0,-1){10}}
\put(42,13){$\tau_1$}

\put(40,15){\circle*{1}}
\put(52,25){\circle*{0.5}}
\put(54,25){\circle*{0.5}}
\put(56,25){\circle*{0.5}}
\put(60,25){\line(1,0){15}}
\put(65,25){\circle*{1}}
\put(65,25){\line(0,-1){10}}
\put(65,15){\circle*{1}}
\put(75,25){\circle*{1}}
\put(75,25){\vector(1,1){5}}
\put(71,27){$\delta_g$}
\put(75,25){\line(0,-1){15}}
\put(75,10){\circle*{1}}
\put(77,8){$\tau_g$}
\end{picture}
$$
\caption{Case 5}
\label{fig:Case_5}
\end{figure}
As in the previous case one has
\begin{align*}
m_{\delta_g\delta} &= k m_{\delta_g\tau_0} = k(m_{\delta_g\sigma_0}+\b{\beta}_0)\,,\\
m_{\delta_g\tau_j} &= N_{j-1}\cdots N_1 k
m_{\delta_g\sigma_0}+ \b{\beta}_j\,,
\\
m_{\delta_g\delta_j} &= N_{j}\cdots N_1 k
m_{\delta_g\sigma_0}+ N_j \b{\beta}_j
= N_{j} m_{\delta_g\tau_j}
\end{align*}
for $j=1,\ldots, g$.
The expression for $B(t)$ implies that $r=g+1$, $m_j = m_{\delta_g\tau_j}$ for $j\ge 0$;
$n_1 = m_{\delta_g\delta}$ and $n_j = m_{\delta_g\delta_{j-1}}$ for $j\ge 2$. Thus the integers
$\seq{m}1r$ permit to compute the intergers $\seq{\b{\beta}}0g$.

\begin{remark}
The characterizations of the different possibilities for the location of the vertex $\delta$ are already made
in the course of the analysis of the different cases above. For a convenience of understanding let us
summarize these characterizations. Let $\ell = \w C\cdot D_{\sigma_0}$ be the intersection multiplicity
between $\w C$ and $D_{\sigma_0}$ and let $\mu:= m_1 - m_{\delta_g\sigma_0}$. Then one has
$$
\begin{array}{cccccc}
Case & \ell & \mu & \b{\beta_0} &  \b{\beta}_1 & \\
1 & 1 &  & 1 & \\
2 & > 1 & 1 & 1 &  \\
3 & > 1 & \mu > \ell & \ell &  \mu \\
4 & > 1 & \mu <  \ell \mbox{ and }   \mu \not| \ell \ & \mu & \ell \\
5 & > 1 & \mu <\ell \mbox{ and } \mu | \ell & \mu &   \\
\end{array}
$$
\end{remark}
\end{proof}

\begin{remark}
 The Poincar\'e series of a curve (reducible or irreducible) on the $E_8$ surface singularity being either
 the Alexander polynomial of the corresponding link or the Alexander polynomial divided by $1-t$, is a topological
 invariant of the curve. Therefore Theorem~\ref{theo:One_curve} means that for $i\ne j$, $i,j\le 7$, all arcs from
 the component $\calE_i$ are not topologically equivalent to arcs from the component $\calE_j$ (except those from
 the intersection of the components).
\end{remark}

\section{The Poincar\'e polynomial of one divisorial valuation and the topological type}\label{sec:One_divisorial}
Let $v$ be a divisorial valuation defined by a component $D_{\sigma^*}$ of the exceptional divisor $\DD$ of a
resolution $\pi:(\XX,\DD)\to(S,0)$ of the $E_8$ surface singularity $(S,0)$.
We assume that $\pi$ is the minimal
modification containing the component defining the valuation, i.~e., the minimal resolution of the valuation $v$.
Let $\pi':(\XX',\DD')\to(S,0)$ be the resolution of the surface described above and let $D_{\sigma_0}$ be the
component of the exceptional divisor $\DD'$ which either coincides with $D_{\sigma^*}$ or is such that $D_{\sigma^*}$
is born by a sequence of blow-ups starting at a point of $D_{\sigma_0}$ (smooth in $\DD'$). Assume that
$\sigma_0\ne 8$. This means that the component $D_{\sigma^*}$ does not originate from a point of $D_8$
(smooth in $\DD'$).

\begin{theorem}\label{theo:one_divisorial}
 In the described situation the
 Poincar\'e series $P_v(t)$ of the divisorial valuation $v$ determines the
 combinatorial type of the minimal resolution of the valuation.
\end{theorem}

\begin{proof}
One has the following analogue of Lemma~\ref{lemma:curveD'}.

 \begin{lemma}\label{lemma:divisorD'}
  The Poincar\'e series $P_v(t)$ of the divisorial valuation $v$ determines the resolution $\pi'$ and the component
  $D_{\sigma_0}$ in $\DD'$.
 \end{lemma}

 \begin{proof}
  The proof is essentially the same as of Lemma~\ref{lemma:curveD'}. The component $D_{\sigma_0}$ coincides with
  $D_1$ if and only if either the Poincar\'e series $P_v(t)$ contains less than two binomial factors with
  the exponent $(-1)$ or $m_2/m_1> 2$. (Again, as in the curve case, the first option does not take place,
  but it is easier to avoid a proof of that.) If $D_{\sigma_0}\ne D_1$, the Poincar\'e series $P_v(t)$ contains
  at least two binomial factors with the exponent $(-1)$ and one has $m_1=\ell m_{\sigma_0 8}$,
  $m_2=\ell m_{\sigma_0 1}$, where $\ell$ is the intersection number of the strict transform in $\XX'$ of a curvette
  at the component defining the valuation with the component $D_{\sigma_0}$. One has $1<m_2/m_1<2$.
  If $m_2/m_1\ne 5/3$, the ratio $m_2/m_1$ determines the component $D_{\sigma_0}$. The ratio $m_2/m_1$
  is equal to $5/3$ if and only if the component $D_{\sigma_0}$ either coincides with $D_3$, or coincides
  with $D_4$, or is produced by blow-ups inbetween $D_3$ and $D_4$. If $D_{\sigma_0}\ne D_4$, the Poincar\'e series
  contains at least three binomial factors with the exponent $-1$ and the ratio $m_3:m_2:m_1$ determines the
  component $D_{\sigma_0}$ (Lemma~\ref{lemma:planar_graph}). If the Poincar\'e series contains less than
  three binomial factors with the exponent $-1$ or the ratio $m_3:m_2:m_1$ is different from $8:5:3$, one has
  $D_{\sigma_0}=D_4$.
 \end{proof}

  The vertices $\sigma^*$ and $\sigma_0$ coincide if and only if
  $$
  P_v(t)=(1-t^{m_{1\sigma_0}})^{-1}(1-t^{m_{4\sigma_0}})^{-1}(1-t^{m_{8\sigma_0}})^{-1}(1-t^{m_{3\sigma_0}})\,.
  $$
  Let $\sigma^*\neq \sigma_0$ and let $\pi'': (\XX, \DD'')\to (\XX',P)$ be the minimal modification of $\XX'$
  containing the component $D_{\sigma^*}$ defining the valuation $v$ ($P\in D_{\sigma_0}$).
  Let $(\ell_{\sigma^*},P)\subset(\XX',P)$ be a curvette at the component $D_{\sigma^*}$.
  The dual graph of the modification $\pi''$ differs from the dual graph of the minimal resolution of
  the curve $\widetilde{C}=\ell_{\sigma^*}$ by a tail
  of length $k\ge 0$ attached to the vertex $\delta_g$ corresponding to the curve $\widetilde{C}$:
  see Figure~\ref{fig:Graph_for_a_divisor}.
\begin{figure}[h]
$$
\unitlength=1.00mm
\begin{picture}(80.00,20.00)(-10,13)
\thinlines
\put(-5,30){\line(1,0){41}}
\put(44,30){\line(1,0){31}} \put(38,30){\circle*{0.5}}
\put(40,30){\circle*{0.5}} \put(42,30){\circle*{0.5}}
\put(30,10){\line(0,1){20}} \put(50,20){\line(0,1){10}}
\put(60,10){\line(0,1){20}} \put(10,15){\line(0,1){15}}
\thinlines
\put(20,30){\circle*{1}} \put(30,30){\circle*{1}}
\put(50,30){\circle*{1}} \put(60,30){\circle*{1}}
\put(65,30){\circle*{1}} \put(70,30){\circle*{1}}
\put(75,30){\circle*{1}} \put(75,30){\circle{2}}

\put(30,20){\circle*{1}}
\put(60,25){\circle*{1}}
\put(60,20){\circle*{1}}
\put(60,15){\circle*{1}}
\put(10,30){\circle*{1}}
\put(30,10){\circle*{1}} \put(50,20){\circle*{1}}
\put(60,10){\circle*{1}}
\put(-5,30){\circle*{1}}
\put(0,30){\circle*{1}} \put(5,30){\circle*{1}}
\put(15,30){\circle*{1}} \put(25,30){\circle*{1}}
\put(35,30){\circle*{1}} \put(45,30){\circle*{1}}
\put(55,30){\circle*{1}} \put(10,25){\circle*{1}}
\put(10,20){\circle*{1}} \put(10,15){\circle*{1}}
\put(30,25){\circle*{1}} \put(30,15){\circle*{1}}
\put(35,30){\circle*{1}} \put(-9,25){$\tau_0$}
\put(11.5,14){$\tau_1$}
\put(31.5,9){$\tau_2$}
\put(61.5,9){$\tau_g$}
\put(9,32){$\delta_1$} \put(29,32){$\delta_2$}
\put(57.5,33){$\delta_g$}
\put(77,27){$\sigma^*$}
\end{picture}
$$
\caption{Resolution graph for a divisorial valuation.}
\label{fig:Graph_for_a_divisor}
\end{figure}
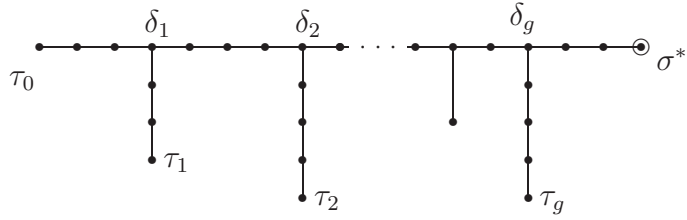
The only difference with the case of an irreducible curve treated in Theorem~\ref{theo:One_curve} above
(and applied to the curve $C=\pi'(\widetilde{C})$) consists in the necessity to find the length $k$ of the tail.
In the case $k=0$ the vertex $\sigma^*$ just coincides with $\delta_g$. In this case the Poincar\'e series
$P_v(t)$ does not contain the factor $(1-t^{m_{\delta_g\delta_g}})$ (since now there is no arrow at the vertex
$\delta_g$). If $k>0$ then, in order to obtain the Poincar\'e series $P_v(t)$, one has to add the factor
$(1-t^{m_{\sigma^*\sigma^*}})^{-1}$ to the decomposition of the Poincar\'e series $P_C(t)$.
In any case one has
$$
P_{v}(t) = P_C(t) (1-t^{m_{\sigma^*\sigma^*}})^{-1}.
$$
The intersection number $\ell = \w C\cdot D_{\sigma_0}$ can be determined from the Poincar\'e series $P_v(t)$
just in the same way as for a curve valuation.

As a consequence, the proof for a divisorial valuation almost repeats the one of Theorem~\ref{theo:One_curve}
for an irreducible curve with the additional duty to determine the component (vertex) $\sigma^*$
and the multiplicity $m_{\sigma^*\sigma^*}$. The analogue of the series $B(t)$ considered in the proof of
Theorem~\ref{theo:One_curve} is the series
$$
B_v(t) = B(t)(1-t^{m_{\sigma^*\sigma^*}})^{-1} =
Q(t)(1-t^{m_{\delta_g\delta}})(1-t^{m_{\sigma^*\sigma^*}})^{-1}.
$$
(Notice that for all $\sigma$ not in the tail one has $m_{\sigma^*\sigma}= m_{\delta_g\sigma}$.)
Let us write the series $B_v(t)$ in the form
$$
B_v (t) = \prod_{k=1}^{r} (1-t^{m_k})^{-1} \ \cdot \
\prod_{k=1}^{r-1} (1-t^{n_k})
$$
with $m_1<\cdots < m_r$ and $n_1<\cdots < n_{r-1}$.

{\bf Case 2.} The component of the exceptional divisor corresponding to the vertex $\sigma^*$ is produced by $k$
supplementary blow-ups at smooth points starting at a point of $\delta_g$.
If $k=0$, i.~e., if $\sigma^*=\delta_g = \delta$, the series $B_v(t)$ consists only of one term
$(1-t^{m_{\sigma^*\tau_0}})^{-1}$. Otherwise
$$
B_v(t) = (1-t^{m_1})^{-1} (1-t^{m_2})^{-1} (1-t^{n_1})
$$
and $k=m_2-n_1$.

{\bf Cases 3, 4, 5.}
From the proof of Theorem 1, it follows that the integers $\seq{m}1{i}$ alongside with
$\b{\beta}_0$ and $\b{\beta}_1$ permit to determine the numbers $\seq{\b{\beta}}0i$. For $i=1,\ldots, r$ let
$\varepsilon_i := \gcd(\b{\beta}_0, \b{\beta}_1, m_1,\ldots, m_i)$.
One can see that $k=0$ if and only if $\varepsilon_r < \varepsilon_{r-1}$. In this case $\sigma^* = \delta_g$.
Otherwise one has $k=m_r-n_{r-1}$, $m_{\sigma^*\sigma^*}=m_{\delta_g\delta_g}+k$.
 \end{proof}

\section{The Poincar\'e polynomial of a collection of divisorial valuations and the topological type}
\label{sec:divisorial_collection}
Let $v_i$, $i=1,\ldots,r$, be divisorial valuations defined by components of the exceptional divisor $\DD$ of a
resolution $\pi:(\XX,\DD)\to(S,0)$ of the $E_8$ surface singularity $(S,0)$.
We assume that $\pi$ is the minimal modification containing the components defining the valuations, i.~e.,
the minimal resolution of the collection $\{v_i\}$ of valuations.
The resolution $\pi$ can be obtained from the minimal resolution of the $E_8$ surface singularity $(S,0)$
by a sequence of blow-ups such that at first some of them are made at intersection points of the components
of the exceptional divisor (and produce a resolution $\pi':(\XX',\DD')\to(S,0)$ with a ``three tails''
dual graph) and later additional blow-ups do not touch the intersection points of the components of $\DD'$,
but start from smooth points of $\DD'$. Assume that the modification $\pi$ of $\pi'$ does not include
blow-ups of smooth (in $\DD'$) points of the component $D_8$.

\begin{theorem}\label{theo:divisorial_collection}
 In the described situation the Poincar\'e series $P_{\{v_i\}}(\tt)$, $\tt=(t_1, \ldots, t_r)$, of the collection
 $\{v_i\}$ of divisorial valuations determines the combinatorial type of the minimal resolution of collection.
\end{theorem}

\begin{proof}
Equation~\ref{eq:ACampo-divis} implies the following {\it projection formula\/}:
if $\{i_1,\ldots,i_{\ell}\}\subset\{1,\ldots,r\}$, then the Poincar\'e series of
the $\ell$-index filtration corresponding to the divisorial valuations $v_{i_1}$,\dots, $v_{i_{\ell}}$
is obtained from the Poincar\'e series $P_{\{v_i\}}(t_1,\ldots, t_r)$ by substituting the
variables $t_i$ with $i\notin \{i_1,\ldots,i_{\ell}\}$ by $1$.

\begin{remark}
The last property does not hold for the Poincar\'e series of the filtration defined by a collection
of curve valuations. This makes the proof of the corresponding statement for curves valuations
(Theorem~\ref{theo:curve_collection} below) somewhat more complicated.
\end{remark}

The dual graph of the minimal resolution of a set of divisorial valuations is determined by the dual
graph of the minimal resolution for each divisor plus the deviation points of the resolutions for each pair
of divisors. The projection formula alongside with Theorem~\ref{theo:one_divisorial} imply that
the Poincar\'e series $P_{\{v_i\}}(\tt)$ determines the minimal resolution graph of each valuation from
the collection and, in particular, the component of the exceptional divisor $\DD'$ of the modification $\pi'$
which the resolution of the divisorial valuation start from. If, for two divisorial valuations from the collection,
these starting components are different, one does not need to find the deviation point. Assume that the starting
components coincide. In order to find the deviation point, without lost of generality (due to the projection formula)
one may assume that $r=2$, i.~e., that the collection consists of these two valuations: $v_1$ and $v_2$.
In this situation one has the following picture. In the dual resolution graph on the geodesic inbetween
the vertices $\sigma_1$ and $\sigma_2$ (defining the divisorial valuation) the ratio
$m_{\sigma\sigma_1}/m_{\sigma\sigma_2}$ (as a function on $\sigma$) is strictly monotonous
being maximal at the vertex $\sigma_1$ and minimal at $\sigma_2$;
on the components of the closure of the complement to this geodesic in the dual graph
this ratio is constant: see~\cite[page 43]{FAOM}, see also Proposition~\ref{prop:P1}
below for the same statemant in a somewhat different setting.
One of the components includes all the vertices corresponding to the components of the
exceptional divisor $\DD'$ and, in particular, the vertex $8$.
The factor $(1-t_1^{m_{8\sigma_1}}t_2^{m_{8\sigma_2}})^{-1}$ participates in the decomposition of the Poincar\'e
series $P_{\{v_1, v_2\}}(t_1, t_2)$. Moreover, its exponent $(m_{8\sigma_1},m_{8\sigma_2})$ is the minimal one in it.
The splitting point of the resolutions of the valuations $v_1$ and $v_2$ also participates in the decomposition
and is maximal in the described components, i.~e., among those with
$m_{\sigma\sigma_1}/m_{\sigma\sigma_2}=m_{8,\sigma_1}/m_{8\sigma_2}$. This determines the splitting point between
the resolutions of the valuations $v_1$ and $v_2$.
\end{proof}

\section{The Poincar\'e polynomial of a reducible curve and the topological type}\label{sec:curve_collection}

Let $C=\bigcup_{i=1}^r C_i$ be a (reducible: $r>1$) curve germ on the surface $(S,0)$ and let
$\pi:(\XX,\DD)\to(S,0)$ be the minimal embedded resolution of the curve $(C,0)\subset(S,0)$. Let us assume
that the resolution process does not contain a blow-up of a smooth (in the exceptional divisor) point of
the component $D_8$ of the minimal resolution of the surface $(S,0)$.

 \begin{theorem}\label{theo:curve_collection}
In the described situation the Poincar\'e series $P_{C}(\tt)$,
$\tt=(t_1, \ldots, t_r)$, of the curve $C = \bigcup_{i=1}^r C_i$
determines the combinatorial type of the
minimal resolution of the curve.
\end{theorem}

\begin{proof}
We have to show that the Poincar\'e series $P_C(\tt)$ determines the minimal
resolution graph $\Gamma$ of $C$. In the case under consideration one has a projection formula different
of the one for divisorial valuations.

In what follows, let us denote $m_{\sigma\sigma(i)}$ ($\sigma(i)$ is
the vertex of $\Gamma$ such that the component $D_{\sigma(i)}$ of the exceptional divisor $\DD$ intersects
the strict transform $\widetilde{C}_i$ of the curve $C_i$) by $m_{\sigma}^i$. Therefore one has
$\mm_{\sigma}=(m_{\sigma}^1, \ldots, m_{\sigma}^r)$. The reason (somewhat psychological) for that is the fact that,
for a multi-exponent of a term of the Poincar\'e series $P_C(\seq t1r)$ or of a factor of its decomposition, one
knows its components $m_{\sigma}^1$, \dots, $m_{\sigma}^r$, but does not know the vertex $\sigma$.
One can say that our aim is to find vertices $\sigma(i)$ corresponding to the curve.

Let $i_0\in \{1,\ldots,r\}$. The A'Campo type formula~(\ref{eq:ACampo}) for $P_{C}(\tt)$ implies that
\begin{equation}
\label{eq:E1}
P_C(\tt)_{|_{t_{i_0}=1}} = P_{C\setminus \{C_{i_0}\}}(t_1, \ldots, t_{i_0-1},
t_{i_0 +1}, \ldots, t_r) \cdot (1-\tt^{\mm_{\sigma(i_0)}})_{|_{t_{i_0}=1}} \; .
\end{equation}
Applying (\ref{eq:E1}) several times one gets
\begin{equation}
\label{eq:E2}
P_C(\tt)_{\vert_{t_j=1\text{ for }j\ne i_0}} = P_{C_{i_0}}(t_{i_0}) \cdot \prod_{i\neq i_0}
(1-t_{i_0}^{m_{\sigma(i)}^{i_0}}) \; .
\end{equation}
Pay attention to the fact that
$m_{\sigma(i)}^{i_0} = m_{\sigma(i_0)}^{i}$ and therefore the series
$P_{C_{i_0}}(t_{i_0})$ can be determined from the Poincar\'e series $P_C(\tt)$ if one knows the multiplicity
$\mm_{\sigma(i_0)}$. The strategy of the proof follows the steps from~\cite{FAOM} (see also~\cite{MNach}):

\begin{enumerate}
\item[1)] To detect an index $i_0$ for which one can find the corresponding multiplicity $\mm_{\sigma(i_0)}$
from the A'Campo type formula for $P_C(\tt)$. Then Theorem 2 and equation~(\ref{eq:E2}) permit to recover the
minimal resolution graph $\Gamma_{i_0}$ of the curve
$C_{i_0}$. Equation~(\ref{eq:E1}) gives the possibility to compute the Poincar\'e series
$P_{C\setminus \{C_{i_0}\}}(t_1, \ldots, t_{i_0-1}, t_{i_0 +1}, \ldots, t_r)$ of the curve $C\setminus \{C_{i_0}\}$.
By induction one can assume that the resolution graph $\Gamma^{i_0}$ of the curve
$C\setminus \{C_{i_0}\}$ is known.

\item[2)] To determine the separation vertex of the curves $C_{i_0}$ and $C_j$ for $j\neq i_0$ in order
to join the graphs $\Gamma_{i_0}$ and $\Gamma^{i_0}$ to obtain the resolution graph $\Gamma$.
\end{enumerate}

Once we finish the first step, the second one almost repeats the same steps in the proof of
Theorem~\ref{theo:divisorial_collection} (for divisorial valuations). Therefore we omit the analysis of 2).

Let $[\sigma(j), \sigma(i)]\subset \Gamma$ be the (oriented) geodesic from
$\sigma(j)$ to $\sigma(i)$ and let
$\{\Delta_p\}$, $p\in\Pi$, be the connected components of
$\Gamma\setminus [\sigma(j),\sigma(i)]$. For each $p\in\Pi$ there exists an
unique $\rho_p\in [\sigma(j),\sigma(i)]$ connecting $\Delta_p$ with
$[\sigma(j),\sigma(i)]$, i.~e., such that
$\Delta^*_{p} = \Delta_p\cup \{\rho_p\}$ is connected.

\begin{proposition}\label{prop:P1}
Let $q : \Gamma\to \Q$ be the function defined by
$q(\alpha) = m_{\alpha}^j/m_{\alpha}^i$ for $\alpha\in \Gamma$. Then one has:
\begin{enumerate}
\item The function $q$ is strictly decreasing along the geodesic
$[\sigma(j), \sigma(i)]$.
\item For each $p\in\Pi$, the function $q$ is constant on $\Delta^*_p$.
\end{enumerate}
\end{proposition}

\begin{proof}
Let $\b{C}_k$ ($k=1,\ldots,r$) be the total transform of the curve $C_k$ in $\XX$. One has
$$
\b{C}_k = \w{C}_k + \sum_{\sigma\in \Gamma} m_{\sigma}^k D_\sigma\;,
$$
where $\w{C}_k$ is the strict transform of the curve $C_k$.
For each component $D_\alpha$, $\alpha\in \Gamma$, one has $\b{C}_k\cdot D_{\alpha}=0$
and therefore
\begin{equation}\label{eq:M}
\w{C}_k \cdot D_{\alpha} + \sum_{\sigma\in \Gamma} m_{\sigma}^k \; D_{\sigma}\cdot D_{\alpha} = 0.
\end{equation}
This equation is a consequence of the Mumford formula (see~\cite[Equation (1)]{MUM}) applied to the
function defining the curve $C_k$.
(Let us recall that on the $E_8$ surface singularity $(S,0)$ all divisors are Cartier ones.)

\begin{lemma}\label{lemma:4}
Let $D_{\alpha}$ be a component of the exceptional divisor $\DD$ such that $\w{C}_i\cdot D_{\alpha}=0$
and let $\{\seq {\rho}1s\}\subset \Gamma$
be the set of all vertices connected by an edge with $\alpha$.
Let us assume that either $\w{C}_j$ intersects $D_{\alpha}$ or
there exists $\rho_{i_0}$ such that $q(\rho_{i_0}) > q(\alpha)$. Then
there exists $\rho_k$ such that $q(\alpha)> q(\rho_k)$.
\end{lemma}

\begin{proof}
Assume that $q(\rho_k)\ge q(\alpha)$ for any $k=1,\ldots, s$. Applying (\ref{eq:M}) to
$C_j$ and $C_i$ one gets:
\begin{align*}
0 &= \w{C}_j \cdot D_{\alpha} + m_{\alpha}^i D_{\alpha}^2 + \sum_{k=1}^s m_{\rho_k}^j
\ge
\\
&\ge  \w{C}_j \cdot D_{\alpha} + m_{\alpha}^i D_{\alpha}^2 + \sum_{k=1}^s
q(\alpha) m_{\rho_k}^i \ge \\
& =  \w{C}_j \cdot D_{\alpha} + q(\alpha) ( m_{\alpha}^j D_{\alpha}^2 + \sum_{k=1}^s
m_{\rho_k}^i) = {\widetilde{C}}_j \cdot D_{\alpha} \ge 0
\end{align*}
The inequality is strict if $\w{C}_j \cdot D_{\alpha}>0$ or if
there exists $i_0$ such that $q(\rho_{i_0})>q(\alpha)$. This implies the statement.
\end{proof}

Let $\alpha$ and $\beta$ be two vertices of $\Gamma$ connected by an edge and let
$q(\alpha)>q(\beta)$. Lemma~\ref{lemma:4} permits to construct a maximal
sequence $\alpha_0, \alpha_1, \ldots, \alpha_k$ of consecutive vertices starting with $\alpha$ and
$\beta$ (i.~e., $\alpha_0 = \alpha$, $\alpha_1=\beta$) such that $q(\alpha_i)>q(\alpha_{i+1})$.
(We will call a sequence of this sort {\em a decreasing path}. If the inequality is in the other direction,
the path will be called {\em increasing}.)
The maximality means that either $\alpha_k$ is a deadend of $\Gamma$ or $\w{C}_i\cdot D_{\alpha_k}\neq 0$.
If $\alpha_k$ is a deadend, $\alpha_{k-1}$ is the only vertex connected with $\alpha_k$ and
Lemma~\ref{lemma:4} implies that $q(\alpha_k)=q(\alpha_{k-1})$. Therefore the
constructed path finishes by the vertex $\alpha_k=\sigma(i)$.
Note that, if $\alpha\in [\sigma(j), \sigma(i)]$ and $\beta\notin [\sigma(j),\sigma(i)]$, the end of
a maximal decreasing (or increasing) path has to finish at a deadend and therefore $q(\alpha)=q(\beta)$.
In particular, this implies that
the function $q$ is constant on each connected set $\Delta^*_p$.

Assume that $\sigma(i)\neq \sigma(j)$. Lemma~\ref{lemma:4} implies that there exists
a vertex $\alpha_1$ connected with $\sigma(j)$ such that
$q(\sigma(j)) > q(\alpha_1)$. Therefore the maximal decreasing path starting with
$\sigma(j)$ and $\alpha_1$ coincides with the geodesic $[\sigma(j), \sigma(i)]$.
\end{proof}

Proposition~\ref{prop:P1} implies that, for any fixed $i_0$ and
for any $j\neq i_0$ and $\sigma\in \Gamma$, one has
$m_{\sigma}^j/m_{\sigma}^{i_0} \ge m_{\sigma(i_0)}^j/m_{\sigma(i_0)}^{i_0}$. Therefore one has
$$
\frac {1}{m_{\sigma}^{i_0}} \mm_{\sigma} \ge
\frac {1}{m_{\sigma(i_0)}^{i_0}} \mm_{\sigma(i_0)} \; .
$$

Let $P_C(\tt) = \prod_{k=1}^{p} (1-\tt^{\nn_k})^{s_k}$ be the Poincar\'e series of the curve
$C$, where $s_k\neq 0$ for all $k$. For $i\in \{1,\ldots,r\}$ let
$k=k(i)$ be such that
$$
\frac{1}{n_{j}^{i}} \nn_{j} \ge
\frac{1}{n_{k}^{i}} \nn_{k}
$$
for all $j$.
Let $E\subset \{1,\ldots, p\}$ be the set of indices $k$ such that
$k=k(i)$ for some $i\in \{1,\ldots, r\}$ and for $k\in E$ let $A(k)\subset
\{1,\ldots, r\}$ denote the set of indices $i$ such that $k=k(i)$. Note that
$A(k)$ contains all the indices $i\in \{1,\ldots,r\}$ such that
$\nn_k = \mm_{\sigma(i)}$. Let $B(k)$ be the subset of such indices.
Our aim is to show that one can find $k\in E$ such that $B(k)\neq \emptyset$.

Let $j\in A(k)$, $j\notin B(k)$. One has
$$
\frac {1}{n_{k}^{j}}\nn_{k} > \frac{1}{m_{\sigma(j)}^{j}} \mm_{\sigma(j)}
$$
and therefore $\chi(\oD_{\sigma(j)})=0$. This
implies that $\sigma(j)$ is connected with only one vertex in $\Gamma$ (plus
the  arrow corresponding to $\w{C}_j$), i.~e., $\sigma(j)$ is a deadend of the
resolution graph of the curve $C\setminus\{C_j\}$. In particular, there are at most two indices
$i,j \in A(k)$ such that $\mm_{\sigma(i)}$ and $\mm_{\sigma(j)}$ are different from
$\nn_k$. Moreover, if there are two indices of this sort, the vertex $\sigma\in
\Gamma$ such that
$\nn_k = \mm_{\sigma}$ is the vertex 3 corresponding to the divisor
$D_3$ of the minimal resolution of $(S,0)$.
In fact in this case the strict transforms
$\w{C}_i$ and $\w{C_j}$ are curvettes at the divisors $D_1$ and $D_4$.
Therefore, if $\# A(k)\ge 3$, there exists $i_0\in B(k)$.

Let $k\in E$ be such that $B(k)=\emptyset$ and let
us assume that $\nn_k = \mm_3$ (i.~e., that the multiplicity $\nn_k$ is the
multiplicity of the divisor $D_3$).
Let $\mm_{8} = (m_8^1, \ldots, m_8^r)$ be the multiplicity of the divisor $D_8$.
Notice that $\mm_8$ is determined by the Poincar\'e series $P_C(\tt)$ because
the decomposition of the Poincar\'e series contains the factor $(1-\tt^{\mm_8})^{-1}$ and, moreover,
the multiplicity $\mm_8$  is the smallest one appearing in it.

If there exists $i\in A(k)$ such that $\w{C_i}$ is a curvete at $D_1$, then $m_8^i = 2$ (note that
$m_8^i=2$ implies that $\w{C}_i$ is a curvette at $D_1$). In this case, if $A(k)\neq \{1,\ldots, r\}$,
one can choose any other $k'\in E$ instead of $k$.
If $A(k) = \{1,\ldots, r\}$, then  $r=2$ and the branches $C_1$ and $C_2$ are
curvettes at the divisors $D_1$ and $D_4$ (see Figure~\ref{fig:Bad_case}).
This situation is equivalent to have the Poincar\'e series of the form
$P_C(t_1,t_2) = (1-\tt^{(2,3)})^{-1}(1-\tt^{(10,15)})$, what gives the statement in
this case. Note also that in this case $m_8^2=3$.
If $A(k)=\{i\}$ and $\w C_i$ is a curvette at the divisor $D_4$ then
one has $m_8^i=3$. However this condition does not characterize completely the
situation described: for $\nn_{k'}=\mm_{7}$ and $A(k')=\{j\}$ with
$\w C_j$ a curvette at $D_7$ one has also that
$m_8^j=3$. If the both multiplicities
appear simultaneously, one can distinguish the first one because
$\nn_{k'} = \mm_{7}$ is always a multiple of $\mm_{8}$ (see Proposition~\ref{prop:P1}) but
$\nn_k$ is not (in the presence of $k'$). This permits to determine the index $k$ in
this case from the information given by the series $P_C(\tt)$.

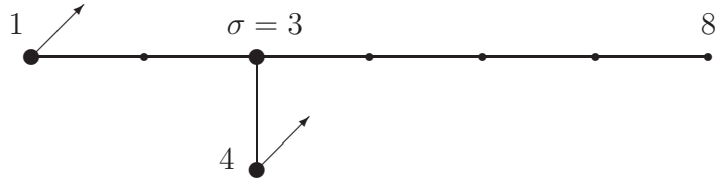
\begin{figure}[h]
$$
\unitlength=1.0mm
\begin{picture}(80.00,25.00)(-10,15)
\thinlines
\put(-15,30){\line(1,0){90}}
\put(-15,30){\vector(1,1){7}}
\put(-15,30){\circle*{2}}
\put(0,30){\circle*{1}}
\put(15,30){\circle*{2}}
\put(30,30){\circle*{1}}
\put(45,30){\circle*{1}}
\put(60,30){\circle*{1}}
\put(75,30){\circle*{1}}
\put(15,30){\line(0,-1){15}}
\put(15,15){\circle*{2}}
\put(15,15){\vector(1,1){7}}
\put(-18,33){{$1$}}
\put(11,33){{$\sigma = 3$}}
\put(74,33){{$8$}}
\put(10,15){{$4$}}
\end{picture}
$$
\caption{The case $\nn_k=\mm_3$, $r=2$, $B(k)=\emptyset$}
\label{fig:Bad_case}
\end{figure}

Let us now consider the case when one has $B(k) = \emptyset$ and
$\nn_k = \mm_{\sigma} \neq \mm_3$ for some $\sigma\in \Gamma$. In this case one has
$A(k)=\{i\}$ and
$\sigma(i)$ is a deadend of the dual resolution graph of the curve $C\setminus \{C_i\}$. In particular,
the vertex $\sigma$ appears after $\sigma(i)$ in the resolution process of a certain branch $C_j$,
$j\neq i$, which is not a curvette at $D_{\sigma}$. It is clear that in this case $\nn_k <
\mm_{\sigma(j)}$ and also $\nn_k < \nn_{k(j)}$. Thus in this case we
take $k'=k(j)$ and $\nn_{k'}$ instead of $k$ and $\nn_k$.
Iterating this procedure one gets $k'$
such that $B(k')\neq \emptyset$. Note that this situation can be determined from
$P_C(\tt)$ taking $k\in E$ such that $\nn_k$ is maximal among the elements $\nn_k$
for $k\in E$ not excluded on the previous stages.

Once we have an index $k\in E$ such that $B(k)\neq \emptyset$ we have to choose an
index $i_0\in B(k)$. Since $n_k^i > n_k^j$ for $i\in B(k)$ and $j\in A(k)\setminus B(k)$ and
$n_k^i= n_k^j$ for $i,j\in B(k)$, for the role of $i_0$ one can take an index from $A(k)$ such that
$n_k^{i_0}$ is the maximal one in $\{n_k^i : i\in A(k)\}$.
This finishes the step 1) of the proof and thus the proof itself.
\end{proof}

Addresses:

A. Campillo and F. Delgado:
IMUVA (Instituto de Investigaci\'on en
Matem\'aticas), Universidad de Valladolid,
Paseo de Bel\'en, 7, 47011 Valladolid, Spain.
\newline E-mail: campillo\symbol{'100}agt.uva.es, fdelgado\symbol{'100}agt.uva.es

S.M. Gusein-Zade:
Moscow State University, Faculty of Mathematics and Mechanics, Moscow, GSP-1, 119991, Russia.
\newline E-mail: sabir\symbol{'100}mccme.ru

\end{document}